\documentclass[11pt,a4paper]{article}
\usepackage[utf8]{inputenc}
\usepackage[english]{babel}
\usepackage{amsmath}
\usepackage{amsthm}
\usepackage{amsfonts}
\usepackage{amssymb}
\usepackage{cite}
\usepackage{hyperref}
\usepackage{authblk}
\usepackage{xcolor}
\usepackage{bm}
\usepackage[left=2cm,right=2cm,top=2cm,bottom=2cm]{geometry}

\hypersetup{colorlinks = true, linkcolor={blue},citecolor={blue},urlcolor={blue}}

\allowdisplaybreaks
\numberwithin{equation}{section}

\newtheorem{theorem}{Theorem}[section]

\newtheorem{lemma}{Lemma}[section]

\newtheorem{remark}{Remark}[section]
\theoremstyle{definition}
\newtheorem{definition}{Definition}[section]

\newcommand{\E}{\mathbb{E}}
\renewcommand{\epsilon}{\varepsilon}

\title{\textbf{Stochastic Allen-Cahn Equation with Logarithmic Potential}\footnote{\textbf{Acknowledgment:} The author is very grateful to Luca Scarpa and Carlo Orrieri for their fundamental support throughout the work.}}
\author{Federico Bertacco}
\affil{\small{e-mail: \href{mailto:f.bertacco20@imperial.ac.uk}{f.bertacco20@imperial.ac.uk}}\\
\small{Department of Mathematics, Imperial College London}\\
\small{Huxley Building, South Kensington, London SW7 2AZ}}
\date{\vspace{-8ex}}

\begin{document}
{\hypersetup{urlcolor = {black}} \maketitle}
\begin{abstract}
\noindent
We prove existence and uniqueness of a solution for the stochastic Allen-Cahn equation with logarithmic potential and multiplicative Wiener noise, under homogeneous Neumann boundary condition. The existence of a solution is obtained in the variational sense by means of an approximated equation involving a Yosida regularization of the nonlinearity. The noise is assumed to vanish at the extremal points of the physical relevant domain and to satisfy a suitable Lipschitz-continuity property, allowing to prove uniform estimates of the approximated solution. The passage to the limit is then carried out and continuous dependence on the initial datum of the solution is verified. Under an additional assumption, the existence of an analytically strong solution is proved. Finally, estimates for the derivatives of the logarithmic potential are derived, in view of the study of an optimal control problem associated to the stochastic Allen-Cahn equation.

\vspace*{4mm}
\noindent
\textbf{AMS Subject Classification:} 35K20, 35K55, 35R60, 60H15.

\vspace*{4mm}
\noindent
\textbf{Key words:} stochastic Allen-Cahn equation, logarithmic potential, variational approach, existence and uniqueness, Yosida approximation, multiplicative noise.
\end{abstract}

{\hypersetup{linkcolor=black}
\tableofcontents}

\section{Introduction}
The study of problems where there is an interface evolving in time plays an important role in many current scientific, engineering, and industrial applications. In particular, phase field models have become an increasingly popular tool to model processes involving thin interface layers between almost homogeneous regions. These models are often referred to as ``diffuse-interface models" and describe the dynamic of interfaces by layers of small thickness. The Allen-Cahn equation belongs to this class of models and it is used to describe phase separation and the evolution of interfaces between the phases for systems without mass conservation. The Allen-Cahn equation was firstly introduced within the Van der Waals theory of phase transitions and then it was resumed in \cite{Allen} to model the growth of grains in crystalline materials near their melting points. The unknown process $u$ is called ``non-conserved order parameter" and represents the normalized density of one of the two involved phases. The sets $\{u = 1\}$ and $\{u = -1\}$ correspond to the pure regions, while the interfacial region $\{-1 < u < 1\}$ corresponds to the region where the two phases coexist and are mixed together. For a more comprehensive presentation of phase field models we refer to the monograph \cite{Brokate} and references therein.

\smallskip
The deterministic Allen-Cahn equation can be obtained as the $L^2$-gradient flow of the Ginzburg-Landau free energy associated to the order parameter $u$, which is given by 
\begin{equation}\label{freeEnergy}
\mathcal{E}(u) := \int_D \frac{1}{2} |\nabla u|^2 + F(u)
\end{equation}
Here, the function $F: \mathbb{R} \to \mathbb{R}$ is a so-called double-well potential that energetically prefers (almost) pure phases. Several choices of $F$ have been proposed in the literature, for example possible choices are
\begin{alignat}{2}
& F_{\text{pol}}(r) := \frac{1}{4}(1-r^2)^2, &&  \quad r \in \mathbb{R} \label{polPotential}\\
& F_{\log}(r) := (1+r) \ln(1+r)+(1-r)\ln(1-r)-c r^2, &&  \quad r \in (-1, 1), \quad c > 1 \label{logPotential}
\end{alignat} 
which are usually refereed to as polynomial and logarithmic potential, respectively. The polynomial potential $F_{\text{pol}}$ is certainly easy to handle from the mathematical point of view, however it is only an approximation of the logarithmic potential $F_{\log}$ which is the most relevant choice, in terms of thermodynamical consistency. The function \eqref{logPotential} has two global minima in the interior of the physically relevant domain $[-1,1]$, which is coherent with the physical interpretation of diffuse-interface modeling in which only the values of the variable $u$ in $[-1, 1]$ are meaningful. Let us observe that minimizing the free energy functional $\mathcal{E}$ corresponds to minimize its two components. Therefore candidates for minimizers of \eqref{freeEnergy} prefer on the one hand to take values in one of the two minima of $F$ and on the other hand to have few oscillations between the two pure phases, in order to keep the gradient small.

\smallskip
Given a Lipschitz bounded domain $D \subset \mathbb{R}^d$ ($d = 2, 3$), with boundary $\Gamma$, and a fixed final time $T > 0$, the $L^2$-gradient flow associated to the free energy \eqref{freeEnergy} is given by the semilinear parabolic PDE of the form
\begin{equation}\label{gradientFlow}
\partial_t u - \Delta u + F'(u) = 0   \quad \text{ in }  (0, T) \times D
\end{equation}
where $\Delta$ stands for the Laplacian acting on the space variables. 
Adding a given external force $g$ acting on the system, and completing it with homogeneous Neumann boundary condition and an initial condition, we obtain the following problem
\begin{alignat*}{2}
& \partial_t u - \Delta u + F'(u) = g   \quad && \text{ in }  (0, T) \times D\\
& \partial_{{\bf n}} u = 0   \quad && \text{ in }  (0, T) \times \Gamma\\
& u(0) = u_0 \quad && \text{ in }  D
\end{alignat*}
were, ${\bf n}$ denotes the outward normal unit vector on $\Gamma$. Here the boundary condition can be interpreted as a null interaction of the phase-transition with the hard walls.

\smallskip
The mathematical literature on the deterministic Allen-Cahn equation is extremely developed. Recently, physicists have introduced the so-called dynamic boundary conditions in order to account also for possible interaction with the walls in a confined system. In this framework, let us mention e.g.\ \cite{Orrieri9, Orrieri12, Orrieri16} and references therein where well-posedness and optimal control problems are studied.

\smallskip
While the deterministic Allen-Cahn equation provides a good description of the evolution of the phase separation, on the other hand it is not effective in taking into account the effects due to the thermal fluctuations of the system. These can be accounted by adding a Wiener type noise in the equation itself, hence obtaining the stochastic Allen-Cahn equation, which reads
\begin{equation}\label{stochasticGradientFlow}
du(t) -\Delta u(t) dt +F'(u(t))dt = g(t) dt +B(u(t))dW(t)
\end{equation}
where $W$ is a cylindrical Wiener process defined on a certain separable Hilbert space and $B$ is a suitable stochastically integrable operator with respect to $W$.

\smallskip
The mathematical literature on stochastic Allen-Cahn equations has also been increasingly developed. From the mathematical point of view, the stochastic Allen-Cahn (and similarly Cahn-Hilliard) equation has been studied mainly in the case of polynomial potentials. Concerning the stochastic Cahn-Hilliard equation, one of the first contributions in this direction is \cite{ScarpaDeg25}, in which the authors show existence of solutions via a semigroup approach in the case of polynomial potentials. A more general framework, allowing for rapidly growing potentials, has been analyzed in \cite{ScarpaDeg63, ScarpaDeg64} from a variational approach and in \cite{ScarpaDeg} where the author adopts a technique similar to the one used in this work. The case of logarithmic potentials has also been covered in the works \cite{ScarpaDeg28, ScarpaDeg29} by means of so-called reflection measures. Concerning specifically the stochastic Allen-Cahn equation, we refer here to \cite[Sec.\ 13.4]{DaPrato} in which a comprehensive treatment of the existing literature is discussed. Let us point out \cite{Rockner} and \cite{Manthey} in which well-posedness results in the case of polynomial potentials are presented, \cite{ScarpaDeg1} dealing with unbounded noise, \cite{Orrieri} dealing with dynamic boundary conditions and \cite{ScarpaDeg4} where the noise is considered to be finite-dimensional. Moreover, let us also mention \cite{Cerrai} in which the author shows existence and uniqueness of a mild solution for a general class of reaction-diffusion systems with reaction term having at most polynomial growth, and \cite{Otto} in which the sharp-interface limit of the stochastic Allen-Cahn equation with polynomial potential is taken into consideration. Regarding general well-posedness results for SPDEs, let us mention \cite{Gess} and the references therein, where unique existence of analytical strong solutions for a large class of SPDEs of gradient type is exhibited. However, in that case, a crucial hypothesis used by the author forbids exponential growth of the potential and so our specific case is not contemplated. To the best of our knowledge, the genuine case of logarithmic potential is not presented in the literature and the main novelty contained in this article is the possibility of dealing with the stochastic Allen-Cahn equation in the case of the thermodynamical relevant potential.

\smallskip
We are interested in studying the stochastic partial differential equation \eqref{stochasticGradientFlow} with logarithmic potential $F_{\log}$. The first main difficulty arising in the study of this SPDE is that the derivative of the logarithmic potential blows up at $\pm 1$ preventing us to rely on the classical variational approach by Pardoux, Krylov and Rozovskii (cf. \cite{Pardoux, Krylov, Rockner}). Indeed, in order to study equations of the form \eqref{stochasticGradientFlow} by means of the variational theory, it is necessary that the derivative of $F$ has a polynomially controlled growth, which is not the case when dealing with $F_{\log}$. The second main difficulty is that, in the stochastic scenario, the presence of the noise term in equation \eqref{stochasticGradientFlow} could bring the parameter $u$ outside the physical relevant domain $[-1,1]$ and so outside the domain in which $F_{\log}$ is defined.

\smallskip
The main idea to overcome these problems is to consider a multiplicative Wiener noise, that allows us to bound the parameter $u$, and to smooth out the equation via a Yosida approximation of the singular part. More specifically, the operator $B$, which we suppose to depend on the parameter $u$ itself, is assumed to satisfy a suitable Lipschitz-continuity property and to vanish at the extremal points of the physical relevant domain, i.e.\ at $\pm 1$. Intuitively this means that the noise is ``shut-down" when the parameter $u$ reaches the extremal points $\pm 1$, in order to confine it inside the interval $(-1,1)$. To study \eqref{stochasticGradientFlow} with this type of noise, we consider a regularized equation in which the monotone increasing part of $F'_{\log}$ is substitute by its Yosida approximation. The approximated system satisfies the usual assumptions of the classical variational theory and therefore well-posedness can be retrieved. Then, through monotonicity arguments and using the assumptions on the noise, uniform estimates can be closed and a passage to the limit provides a solution to the original problem. 

\smallskip
Let us briefly present the structure of the paper. We consider the stochastic Allen-Cahn system with homogeneous Neumann boundary condition and a prescribed initial condition in the form
\begin{equation}\label{acSpde}
\begin{alignedat}{2}
& du(t) -\Delta u(t) dt +F'(u(t))dt = g(t) dt +B(u(t))dW(t)  \quad && \text{ in } \quad (0, T) \times D\\
& \partial_{{\bf n}} u = 0   \quad && \text{ in } \quad (0, T) \times \Gamma\\
& u(0) = u_0 \quad && \text{ in } \quad D
\end{alignedat}
\end{equation}
where $W$ is a cylindrical Wiener process on a certain separable Hilbert space $U$ and $B$ takes values in the space of Hilbert-Schmidt operators from $U$ to $L^2(D)$. In Section \ref{notAssRes} we fix the main assumptions that will be used throughout the work, we precise the concept of variational solution and analytically strong solution, and we state the main results. In Section \ref{approxProb} we introduce the approximated problem and we prove its well-posedness. In Section \ref{uniformEstimates} we derive estimates in expectation of the solution to the approximated problem. In Section \ref{proofMainResult} we prove the main results: a variational solution to the original problem is obtained by passing to the limit in suitable topologies; we prove the continuous dependence on the initial datum of the solution; we derive an additional estimate in order to prove the existence of an analytically strong solution; we derive estimates in expectation on the derivatives of the logarithmic potential in view of the study of an optimal control problem associated to the stochastic Allen-Cahn equation.

\section{Notation, assumptions, and main results} \label{notAssRes}
\subsection{Notation}
For any real Banach space $E$, its dual will be denoted by $E^{*}$. The norm in $E$ and the duality paring between $E$ and $E^{*}$ will be denoted by $\|\, \cdot \, \|_E$ and $\langle \, \cdot \, , \, \cdot \, \rangle_E$, respectively. If $(A, \mathcal{A}, \mu)$ is a finite measure space, we denoted with $L^p(A; E)$ the space of $p$-Bochner integrable functions, for any $p \in [1, \infty]$. We will use the symbol $C^{0}([0, T]; E)$ for the space of strongly continuous functions from $[0, T]$ to $E$, with $T > 0$ being a fixed final time. If $E_1$ and $E_2$ are separable Hilbert space, we denote with $\mathcal{L}(E_1, E_2)$ and $\mathcal{L}^2(E_1, E_2)$ the spaces of linear and continuous operators and of Hilbert-Schmidt operators from $E_1$ to $E_2$, respectively. Moreover $(\Omega, \mathcal{F}, (\mathcal{F}_t)_{t \in [0, T]}, \mathbb{P})$ denotes a filtered probability space satisfying the usual conditions (i.e. it is saturated and right continuous) and $W$ is a cylindrical Wiener process valued on a separable Hilbert space $U$. We fix a complete orthonormal system $(u_k)_{k \in \mathbb{N}}$ of $U$. 

\smallskip
As we have anticipated, let $D \subset \mathbb{R}^d$ ($d =2, 3$) be a Lipschitz bounded domain with boundary $\Gamma$ whose outward normal unit vector is denoted with ${\bf n}$. The Lebesgue measure of $D$ will be denoted by $|D|$. We define the following functional spaces
\begin{equation*}
H : = L^2(D), \qquad V : = H^1(D), \qquad V_0 : = \{v \in H^2(D) \, : \, {\bf n} \cdot \nabla v = 0 \text{ a.e.\ on } \Gamma\}
\end{equation*}  
endowed with their natural norms $\| \, \cdot \, \|_H$, $\| \, \cdot \, \|_V$, and $\| \, \cdot \, \|_{V_0}$, respectively. Identifying the Hilbert space $H$ with its dual trough the Riesz-isomorphism theorem, we have the following continuous, dense, and compact inclusions 
\begin{equation*}
V_0 \hookrightarrow V \hookrightarrow H \equiv H^{*} \hookrightarrow V^{*} \hookrightarrow V_0^{*}
\end{equation*}
In particular, we have that $(V, H, V^{*})$ constitutes a Gelfand triple. Let us recall that the Laplace operator, with homogeneous Neumann condition, can be seen as a variational operator
\begin{equation*}
-\Delta : V \to V^{*}, \quad \langle -\Delta u, v \rangle_V : = \int_D \nabla u \nabla v, \quad \forall u, v \in V
\end{equation*}
In the sequel we will use the same symbol $-\Delta$ to denote the Laplace operator intended both as a variational operator and as an operator defined from $V_0$ valued in $H$.

\smallskip
Finally, we will use the symbol $a$ to denote any arbitrary positive constant depending only on the data of the problem, whose value may be updated throughout the proofs. 

\subsection{Assumptions}
We precise here the assumptions that are used throughout the work.

\smallskip
\noindent
\textbf{$\bf{(A1)}$ - Assumptions on the double-well potential.} We assume the potential to be of logarithmic type. Specifically, letting $F_{\log}$ as in \eqref{logPotential}, we define 
\begin{equation*}
F: (-1, 1)\to \mathbb{R}, \quad F(r) := F_{\log}(r)+ K, \quad r \in (-1, 1) 
\end{equation*}
where $K > 0$ is a positive constant such that $F(r)$ is non-negative for all $r \in (-1,1)$. Therefore, we have that 
\begin{equation*}
F'(r) = F'_{\log}(r) = \ln\left(\frac{1+r}{1-r}\right) - 2 c r, \quad r \in (-1, 1)
\end{equation*}
We define 
\begin{equation}
\beta(r) : = \ln\left(\frac{1+r}{1-r}\right), \quad r \in (-1, 1)
\end{equation} 
then $\beta: (-1, 1) \to \mathbb{R}$ is a monotone increasing and continuous function such that $\beta(0)= 0$, and so it can be identified with a maximal monotone graph in $\mathbb{R} \times \mathbb{R}$. 

\smallskip
\noindent
\textbf{$\bf{(A2)}$ - Assumptions on the noise.} We consider a multiplicative type noise with diffusion coefficient $B$. In particular, let $H_1 = \{x \in H \, : \, \|x\|_{L^{\infty}(D)} \leq 1\}$ and consider $(h_k)_{k \in \mathbb{N}} \subset W^{1, \infty}(-1, 1)$ such that for any $k \in \mathbb{N}$, $h_k(\pm 1) = 0$. Then the operator $B : H_1 \to \mathcal{L}^2(U, H)$ is such that
\begin{equation}\label{defBuk}
B(x)u_k := h_k(x), \quad  \forall x \in H_1, \quad \forall k \in \mathbb{N}
\end{equation} 
We assume further that 
\begin{equation}\label{CB}
C_B : = \sum_{k \in \mathbb{N}} \|h_k\|_{W^{1, \infty}(-1, 1)}^2 < \infty
\end{equation}

\smallskip
\noindent
\textbf{$\bf{(A3)}$ - Assumptions on the initial datum.} We assume that 
\begin{equation}\label{condU0}
u_0 \in L^2(\Omega, \mathcal{F}_0;H) \quad \text{ and } \quad  F(u_0) \in L^1(\Omega, \mathcal{F}_0; L^1(D))
\end{equation}
Moreover, we consider the external force $g: [0, T] \times \Omega \to H$ to be progressively measurable and such that 
\begin{equation}\label{condG}
g \in L^2(\Omega; L^2(0, T; H))
\end{equation}
\begin{remark}
The conditions in assumption $(A2)$ ensure that the map
\begin{equation}\label{defB}
B : H_1 \to \mathcal{L}^2(U, H), \quad  B(x)u := \sum_{k \in \mathbb{N}} (u, u_k)_{U} h_k(x), \quad x \in H_1, \quad u \in U
\end{equation}
is well-defined, measurable, and Lipschitz-continuous. Let us first observe that the series in \eqref{defB} is convergent in $H$ for any $x \in H_1$ and $u \in U$, indeed we have
\begin{align*}
\left\|\sum_{k \in \mathbb{N}} (u, u_k)_U h_k(x)\right\|_H & \leq \sum_{k \in \mathbb{N}} |(u,u_k)_U| \, \|h_k(x)\|_H\\
&  \leq \left(\sum_{k \in \mathbb{N}}\|h_k(x)\|_H^2\right)^{1/2} \left(\sum_{k \in \mathbb{N}}|(u, u_k)_U|^2\right)^{1/2}\\
&  \leq (C_B |D|)^{1/2} \|u\|_U < \infty
\end{align*}
From this estimate it also follows that, for any $x \in H_1$, the linear operator $B(x):U \to H$ is continuous. Now we want to show that for every $x \in H_1$ the operator $B(x): U \to H$ is Hilbert-Schmidt, indeed we have  
\begin{equation*}
\|B(x)\|_{\mathcal{L}^2(U, H)}^2 = \sum_{k \in \mathbb{N}} \|B(x)u_k\|_H^2 = \sum_{k \in \mathbb{N}} \|h_k(x)\|_{H}^2 \leq C_B |D|
\end{equation*}
Moreover, let $x$, $y \in H_1$, then 
\begin{equation*}
\|B(x)-B(y)\|_{\mathcal{L}^2(U, H)} = \sum_{k \in \mathbb{N}} \|h_k(x)-h_k(y)\|_H^2 \leq C_B \|x-y\|_H^2
\end{equation*}
and so $B:H_1 \to \mathcal{L}^2(U,H)$ is Lipschitz-continuous.
\end{remark}

Here we also precise the assumptions that we will use to prove the estimates on the derivatives of the logarithmic potential. Let us observe that for every $n \geq 2$ it exits a positive constant $a_n > 0$ such that 
\begin{equation*}
|F^{(n)}(r)| \leq \frac{a_n}{(1-r^2)^{n-1}}, \quad r \in (-1, 1)
\end{equation*}
Therefore, if we define the function $G_n: (-1, 1) \to \mathbb{R}$ as follows
\begin{equation}\label{Gn}
G_n(r) := \frac{1}{(1-r^2)^{n-1}}, \quad r \in (-1, 1)
\end{equation}
deriving estimates on $G_n$ is equivalent to derive estimates on $|F^{(n)}|$. Further, let us also notice that there exit positive constants $a'_n$ and $a''_n$ such that 
\begin{equation} \label{condDerGn}
|G'_n(r)| \leq \frac{a'_n}{(1-r^2)^n} \quad \text{ and } \quad G''_n(r) \leq \frac{a''_n}{(1-r^2)^{n+1}}, \quad r \in (-1,1)
\end{equation}
and so, in particular, it follows that estimating $|G'_n|$ coincides with estimating $G_{n+1}$. We are now ready to state our assumptions.

\smallskip
\noindent
\textbf{$\bf{(A2)_n}$ - Assumptions on the noise.} We assume that the operator $B:H_1 \to \mathcal{L}^2(U, H)$ is still defined as in \eqref{defB} however we require that $(h_k)_{k \in \mathbb{N}} \subset W^{n+1, \infty}(-1,1)$ is a family of functions such that, for any $k \in \mathbb{N}$ and for any $j =0, \dots, n$, it holds that $h^{(j)}_k(\pm 1) = 0$ and 
\begin{equation}\label{Cn}
C_{n} := \sum_{k \in \mathbb{N}} \|h_k\|^2_{W^{n+1, \infty}(-1,1)} < \infty
\end{equation}

\smallskip
\noindent
\textbf{$\bf{(A3)_n}$ - Assumptions on the initial datum.} Beyond assumption $(A3)$ we also assume that the initial datum satisfies 
\begin{equation*}
G_n(u_0) \in L^1(\Omega, \mathcal{F}_0; L^1(D))
\end{equation*}
and that the external force $g$ satisfies
\begin{equation*}
||g||_{L^{\infty}(D)} \leq 1 \quad \text{ a.e.\ in } [0, T] \times \Omega
\end{equation*}

\begin{remark}
Let us observe that for any $n \geq 2$, assumptions $(A2)_n$ and $(A3)_n$ trivially implies assumptions $(A2)$ and $(A3)$.
\end{remark}

\subsection{Main results}
We start by giving the definition of variational solution and analytically strong solution to problem \eqref{acSpde}.
\begin{definition}\label{strongSol}
A variational solution to problem \eqref{acSpde} is a process $u$ such that 
\begin{equation*}
u \in L^2(\Omega; C^0([0,T]; H)) \cap L^2(\Omega; L^2(0, T; V)), \quad F'(u) \in L^2(\Omega; L^2(0, T; H)) 
\end{equation*}
and for all $v \in V$ it holds that 
\begin{equation}\label{varSol}
\begin{split}
\int_D u(t) v  & +\int_0^t \int_D \nabla u(s) \nabla v ds +\int_0^t \int_D F'(u(s)) v ds \\
& =  \int_D u_0 v + \int_0^t \int_D g(s) v ds +  \int_D \left(\int_0^t B(u(s)) dW(s)\right) v
\end{split}
\end{equation}
for every $t \in [0, T]$, $\mathbb{P}$-a.s.
\end{definition}
\begin{definition}\label{analyticalStrongSol}
An analytically strong solution to problem \eqref{acSpde} is a process $u$ such that 
\begin{equation*}
u \in L^2(\Omega; C^0([0,T]; H)) \cap L^2(\Omega; L^2(0, T; V_0)) \cap L^2(\Omega; L^{\infty}(0, T; V)), \quad F'(u) \in L^2(\Omega; L^2(0, T; H)) 
\end{equation*}
and
\begin{equation*}
u(t) - \int_0^t \Delta u(s) ds +\int_0^t F'(u(s)) ds = u_0 + \int_0^t g(s) ds + \int_0^t B(u(s)) dW(s)
\end{equation*}
for every $t \in [0, T]$, $\mathbb{P}$-a.s.
\end{definition}
\noindent
We are now ready to state our main results, namely the following theorems.
\begin{theorem}
Under assumptions $(A1)$-$(A3)$, there exists a unique variational solution to \eqref{acSpde} in the sense of Definition \ref{strongSol}. Furthermore, it exists a constant $a$ such that for any initial data $(u_0^1, g_1)$, $(u_0^2, g_2)$ satisfying assumption $(A3)$, the respective variational solutions $u_1$, $u_2$ of \eqref{acSpde} satisfy 
\begin{equation}\label{continuousDep}
\|u_1-u_2\|_{L^2(\Omega; C^0([0, T]; H)) \cap L^2(\Omega; L^2(0, T; V))} \leq a \left( \|u_0^1-u_0^2\|_{L^2(\Omega; H)} + \|g_1-g_2\|_{L^2(\Omega; L^2(0, T; V^{*}))} \right)
\end{equation}
Moreover, under assumptions $(A1)$-$(A3)$, if the initial datum $u_0 \in L^2(\Omega, \mathcal{F}_0; V)$, then it exists a unique analytically strong solution to \eqref{acSpde} in the sense of Definition \eqref{analyticalStrongSol}.
\end{theorem}
\begin{remark}
Note that \eqref{continuousDep} implies the pathwise uniqueness of the solution. Indeed if $u_0^1 = u_0^2$ and $g_1 = g_2$, then
\begin{equation*}
u_1(t) = u_2(t) \quad \forall t \in [0, T], \quad \mathbb{P}\text{-a.s.}
\end{equation*} 
and so the uniqueness follows.
\end{remark}
\begin{theorem}
Fix $n \geq  2$, under assumptions $(A1)$, $(A2)_n$ and $(A3)_n$, if $u$ denotes the variational solution to \eqref{acSpde}, then it exists a constant $a_n$ such that 
\begin{equation}\label{estimateDerLog}
\|F^{(n)}(u)\|_{L^{\infty}(0, T; L^1(\Omega; D))} + \|F^{(n+1)}(u)\|_{L^1(\Omega; L^1(0, T; D)} \leq a_n
\end{equation}
\end{theorem}

\section{The approximated problem}\label{approxProb}
Let us observe that problem \eqref{acSpde} cannot be studied in the Gelfand triple $(V, H, V^{*})$ by means of the classical variational theory because the drift is not well-defined as a $V^{*}$-valued operator. To overcome this problem we consider a suitable regularized equation and, in this section, we prove its well-posedness.

\smallskip
For any $\lambda \in (0, 1)$ and $r \in \mathbb{R}$, let 
\begin{equation*}
J_{\lambda}(r) := (I+\lambda \beta)^{-1}(r) \quad \text{ and } \quad \beta_\lambda(r) := \lambda^{-1}(I - J_{\lambda})(r) 
\end{equation*}
be the resolvent and the Yosida approximation of the maximal monotone graph $\beta \subset \mathbb{R} \times \mathbb{R}$, respectively. 
\begin{remark}
We state here the main properties of $J_{\lambda}$ and $\beta_{\lambda}$ that we will use in the sequel (for a more comprehensive treatment on the topic we refer to \cite{Barbu}):
\begin{itemize}
\item[-] $J_{\lambda}: \mathbb{R} \to (-1,1)$ is non-expansive, i.e.\ it is Lipschitz-continuous with Lipschitz constant not greater than $1$; $\beta_{\lambda}: \mathbb{R} \to \mathbb{R}$ is Lipschitz-continuous with Lipschitz constant not greater than $\lambda^{-1}$;
\item[-] for any $r \in \mathbb{R}$, $\beta_{\lambda}(r) = \beta(J_{\lambda}(r))$;
\item[-] for any $r \in (-1, 1)$, $\lim_{\lambda \to 0^{+}} J_{\lambda} (r) = r$, and $\lim_{\lambda \to 0^{+}} \beta_{\lambda}(r) = \beta(r)$;
\item[-] for any $r \in \mathbb{R}$, the function $\lambda \mapsto |\beta_{\lambda}(r)|$ is non-decreasing for $\lambda \searrow 0$, i.e.\ $|\beta_{\lambda}(r)| \leq |\beta_{\epsilon}(r)|$ if $\lambda > \epsilon$.
\end{itemize}
\end{remark}
\noindent
We also define the function $\hat{\beta}_{\lambda}: \mathbb{R} \to [0, \infty)$ as $\hat{\beta}_{\lambda}(r):=\int_0^r \beta_{\lambda}(s) ds$, $r \in \mathbb{R}$. With this notation we introduce the approximated potential 
\begin{equation}\label{approxPotential}
F_{\lambda}: \mathbb{R} \to [0, \infty), \quad F_{\lambda}(r) : = F(0) + \hat{\beta}_{\lambda}(r)-c r^2, \quad r \in \mathbb{R}
\end{equation}
We have that $F_{\lambda}$ is non-negative, convex, with quadratic growth, and such that $F_{\lambda}(r) \leq F(r)$, for all $r \in (-1, 1)$. Let $F'_{\lambda}$ be the derivative of $F_{\lambda}$, defined as follows
\begin{equation*}
F'_{\lambda}(r): \mathbb{R}\to \mathbb{R}, \quad F'_{\lambda}(r) := \beta_{\lambda}(r)-2 c r, \quad r \in \mathbb{R}
\end{equation*}
By the properties of the Yosida approximation $\beta_{\lambda}$, it follows that the function $F'_{\lambda}$ is Lipschitz-continuous and we denote its Lipschitz constant with $C_{F'_{\lambda}}$. We also define the operator 
\begin{equation}\label{BetaLambda}
B_{\lambda}:  H \to \mathcal{L}^2(U, H), \quad B_{\lambda}(x) := B(J_{\lambda}(x)) 
\end{equation}
where the operator $B$ is defined as in \eqref{defB}. Then the approximated problem reads
\begin{equation}\label{approximatedProblem}
\begin{alignedat}{2}
& du_{\lambda}(t)-\Delta u_{\lambda}(t) dt + F'_{\lambda}(u_{\lambda}(t)) dt = g(t)dt+B_{\lambda}(u_{\lambda}(t)) dW(t)   \quad && \text{ in } \quad (0, T) \times D  \\
& \partial_{{\bf n}} u_{\lambda} = 0   \quad && \text{ in } \quad (0, T) \times \Gamma \\
& u_{\lambda}(0) = u_0 \quad && \text{ in } \quad D
\end{alignedat}
\end{equation}
We introduce the operator $A_{\lambda} : [0, T] \times \Omega \times V \to V^{*}$ such that for $(t, \omega) \in [0, T]\times \Omega$ and $u \in V$, $A_{\lambda}(t, \omega, u): = -\Delta u + F'_{\lambda}(u)+g(t, \omega)$, or more precisely, for all $v \in V$
\begin{equation}\label{ALambda}
\langle A_{\lambda}(t, \omega, u), v \rangle_V := \int_D \nabla u \cdot \nabla v + \int_D F'_{\lambda}(u) v - \int_D g(t, \omega) v
\end{equation}
It is readily seen that the operator $A_{\lambda}$ is well-defined, i.e.\ that for all $(t, \omega) \in [0, T] \times \Omega$ and $u \in V$, $A_{\lambda}(t, \omega, u)$ is a linear and continuous functional defined on $V$. At this point we can write the approximated equation as 
\begin{equation*}
du_{\lambda}(t) + A_{\lambda}(t, u_{\lambda}(t)) dt = B_{\lambda}(u_{\lambda}(t)) dW(t) \quad \text{ in } \quad (0, T) \times V^{*}
\end{equation*}
The idea now is to study the approximated problem \eqref{approximatedProblem} using the classical variational theory by Pardoux, Krylov and Rozovskii (cf. \cite{Pardoux, Krylov, Rockner}) in the Gelfand triple $(V, H, V^{*})$. We need the following lemma.
\begin{lemma}\label{CoeffHp}
Let $\lambda \in (0,1)$. Then the operator $A_{\lambda}: [0, T] \times \Omega \times V \to V^{*}$ defined in \eqref{ALambda} is progressively measurable, hemicontinuous, weakly monotone, weakly coercive and bounded, i.e.
\begin{itemize}
\item[-] the map 
\begin{equation*}
\eta \mapsto \langle A_{\lambda}(t, \omega, u_1+\eta u_2), v \rangle_V, \quad \eta \in \mathbb{R}
\end{equation*}
is continuous for every $(t, \omega) \in [0, T] \times \Omega$, and $u_1$, $u_2$, $v \in V$;
\item[-] there exists a constant $c > 0$ such that 
\begin{equation*}
\langle A_{\lambda}(t, \omega, u_1)-A_{\lambda}(t, \omega, u_2), u_1-u_2 \rangle_V \geq -c\|u_1-u_2\|_H^2
\end{equation*}
for every $(t, \omega) \in [0, T] \times \Omega$ and $u_1$, $u_2 \in V$;
\item[-] there exist two constants $c_1$, $c_1' > 0$ and an adapted process $f_1 \in L^1([0, T] \times \Omega)$ such that
\begin{equation*}
\langle A_{\lambda}(t, \omega, u), u\rangle_V \geq c_1 \|u\|_V^2-c_1'\|u\|_H^2-f_1(t, \omega)
\end{equation*}
for every $(t, \omega) \in [0, T] \times \Omega$ and $u \in V$;
\item[-] there exist a constant $c_2 > 0$ and an adapted process $f_2 \in L^2([0, T] \times \Omega)$ such that
\begin{equation*}
\|A_{\lambda}(t, \omega, u)\|_{V^{*}} \leq c_2 \|u\|_V+f_2(t, \omega)
\end{equation*}
for every $(t, \omega) \in [0, T] \times \Omega$ and $u \in V$.
\end{itemize}
\end{lemma} 
\begin{proof}
First of all, $A_{\lambda}$ is progressively measurable since so is $g$. Moreover, for every $(t, \omega) \in [0, T]\times \Omega$ and $u_1$, $u_1$, $v \in V$, we have 
\begin{equation*}
 \langle A_{\lambda}(t, \omega, u_1+\eta u_2), v \rangle_V = \int_D \nabla u_1 \cdot \nabla v + \eta  \int_D \nabla u_2 \cdot \nabla v + \int_D F'_{\lambda}(u_1+\eta u_2) v - \int_D g(t, \omega) v
\end{equation*}
which is continuous in $\eta \in \mathbb{R}$ by the Lipschitz-continuity of $F'_{\lambda}$. Secondly, for every $(t, \omega) \in [0, T]\times \Omega$ and $u_1$, $u_1 \in V$, exploiting the Lipschitz-continuity of $F'_{\lambda}$, we have that 
\begin{align*}
& \langle A_{\lambda}(t, \omega, u_1)-A_{\lambda}(t, \omega, u_2), u_1-u_2 \rangle_V \\
& = \int_D |\nabla(u_1-u_2)|^2 +\int_D (F'_{\lambda}(u_1)-F'_{\lambda}(u_2))(u_1-u_2) \\
& \geq - C_{F'_{\lambda}} \|u_1-u_2\|_H^2 \
\end{align*}
from which the weak monotonicity follows with the choice $c = C_{F'_{\lambda}}$. Similarly, for every $(t, \omega) \in [0, T]\times \Omega$ and $u \in V$, by Hölder's inequality and the Lipschitz-continuity of $F'_{\lambda}$ we have that 
\begin{align*}
& \langle A_{\lambda}(t, \omega, u), u\rangle_V \\
& = \int_D |\nabla u|^2 + \int_D F'_{\lambda}(u)u + \int_D g(t, \omega) u \\
& \geq \|\nabla u\|_H^2 - C_{F'_{\lambda}}\|u\|_H^2 - \|g(t, \omega)\|_H \|u\|_H \\
& \geq \|\nabla u\|_H^2  - \left(C_{F'_{\lambda}}+\frac{1}{2}\right)\|u\|^2_H - \frac{1}{2}\|g(t, \omega)\|_H^2 \\
& = \|u\|_V^2 - \left(C_{F'_{\lambda}}+\frac{3}{2}\right)\|u\|^2_H - \frac{1}{2}\|g(t, \omega)\|_H^2 
\end{align*}
so that $c_1 = 1 $, $c'_1 = C_{F'_{\lambda}}+3/2$, and $f_1(t, \omega) = 1/2 \|g(t, \omega)\|_H^2 \in L^1((0, T) \times \Omega)$ is a possible choice. Finally, for every $(t, \omega) \in [0, T]\times \Omega$ and $u$, $v \in V$, by Hölder's inequality and the Lipschitz-continuity of $F'_{\lambda}$, we have that
\begin{align*}
& \langle A_{\lambda}(t, \omega, u), v\rangle_V \\
& = \int_D \nabla u \cdot \nabla v + \int_D F'_{\lambda}(u) v + \int_D g(t, \omega)v \\
& \leq (\|\nabla u\|_H + C_{F'_{\lambda}}\|u\|_H + \|g(t, \omega)\|_H) \|v\|_V
\end{align*}
which implies that
\begin{equation*}
\|A_{\lambda}(t, \omega, u)\|_{V^{*}} \leq (1+C_{F'_{\lambda}})\|u\|_V + \|g(t, \omega)\|_H \\
\end{equation*}
which proves that $A_\lambda$ is also bounded with the choice of coefficient $c_2 = 1+C_{F'_{\lambda}}$ and $f_2(t, \omega) = \|g(t, \omega)\|_H \in L^2((0, T) \times \Omega)$.
\end{proof}
Let us observe that the operator $B_{\lambda}: H \to \mathcal{L}^2(U, H)$ defined in \eqref{BetaLambda} is Lipschitz-continuous. Indeed, exploiting the Lipschitz-continuity of $B$ an that $J_{\lambda}$ is non-expansive, for every $x$, $y \in H$, we have  
\begin{align*}
\|B_{\lambda}(x)-B_{\lambda}(y)\|_{\mathcal{L}^2(U, H)} = \|B(J_{\lambda}(x))- B(J_{\lambda}(y))\|_{\mathcal{L}^2(U, H)} \leq C_B \|J_{\lambda}(x)- J_{\lambda}(y)\|_H \leq C_B \|x-y\|_H
\end{align*}
Therefore Lemma \ref{CoeffHp} and the classical variational theory by Pardoux, Krylov and Rozovskii (cf. \cite{Pardoux, Krylov, Rockner}) ensure that there exists a unique process 
\begin{equation*}
u_{\lambda} \in L^2(\Omega, C^{0}([0, T]; H)) \cap L^2(\Omega, L^2(0, T; V))
\end{equation*}
such that, for any $v \in V$, it holds
\begin{equation}\label{approxEquation}
\begin{split}
\int_D u_{\lambda}(t)v &  + \int_0^t \int_D  \nabla u_{\lambda}(s)\nabla v ds + \int_0^t \int_D F'_{\lambda}(u_{\lambda}(s)) vds \\
& = \int_D u_0 v + \int_0^t \int_D g(s) v ds + \int_D \left(\int_0^t B_{\lambda}(u_{\lambda}(s)) dW(s)\right) v 
\end{split}
\end{equation}
for every $t \in [0, T]$, $\mathbb{P}$-a.s.

\section{Uniform estimates}\label{uniformEstimates}
In this section we derive uniform estimates with respect to $\lambda$. More precisely, in the first and in the second part we prove uniform estimates of $u_{\lambda}$ and $F'_{\lambda}(u_{\lambda})$, while in the last part we show that the sequence $(u_{\lambda})_{\lambda \in (0,1)}$ is a Cauchy sequence.

\subsection{First estimate}
By \cite[Theorem 4.2.5]{Rockner}, Itô's formula for the square of the $H$-norm of $u_{\lambda}$, reads for every $t \in [0, T]$, $\mathbb{P}$-a.s.\
\begin{equation} \label{ulambda1}
\begin{alignedat}{1}
& \frac{1}{2}\|u_{\lambda}(t)\|_H^{2} +\int_0^t \|\nabla u_{\lambda}(s)\|_H^2 ds +\int_0^t (F'_{\lambda}(u_{\lambda}(s)), u_{\lambda}(s))_H ds \\
& = \frac{1}{2} \|u_0\|_H^2 + \int_0^t (u_{\lambda}(s), g(s))_H ds +\frac{1}{2}\int_0^t \|B_{\lambda}(u_{\lambda}(s))\|_{\mathcal{L}^2(U, H)}^2 ds + \int_0^t (u_{\lambda}(s), B_{\lambda}(u_{\lambda}(s)))_H dW(s)
\end{alignedat}
\end{equation}
Since $F'_{\lambda}(u_{\lambda}(s)) = \beta_{\lambda}(u_{\lambda}(s))-2 c u_{\lambda}(s)$ and $\beta_{\lambda}$ is a monotone increasing function such that $\beta_{\lambda}(0) = 0$, then we have 
\begin{equation*}
\int_0^t (F'_{\lambda}(u_{\lambda}(s)), u_{\lambda}(s))_H ds = \int_0^t (\beta_{\lambda}(u_{\lambda}(s)),u_{\lambda}(s))_H ds - 2c \int_0^t \|u_{\lambda}(s)\|_H^2 ds \geq -2c \int_0^t \|u_{\lambda}(s)\|_H^2 ds
\end{equation*} 
Using Hölder's inequality and Young's inequality for the integral containing the external force $g$, we obtain 
\begin{equation*}
\int_0^t (u_{\lambda}(s), g(s))_H ds \leq \frac{1}{2}\int_0^t \|g(s)\|_H^2 ds+\frac{1}{2}\int_0^t \|u_{\lambda}(s)\|_H^2 ds
\end{equation*}
Moreover we have that
\begin{equation*}
\|B_{\lambda}(u_{\lambda}(s))\|_{\mathcal{L}^2(U, H)}^2 = \sum_{k \in \mathbb{N}} \|h_k(J_{\lambda}(u_{\lambda}(s)))\|_H^2 \leq \sum_{k \in \mathbb{N}} \|h_{k}\|_{W^{1, \infty}(-1, 1)}^2 |D| = C_B |D| 
\end{equation*}
and therefore
\begin{equation*}
\frac{1}{2}\int_0^t \|B_{\lambda}(u_{\lambda}(s))\|_{\mathcal{L}^2(U, H)}^2 ds  \leq  \frac{C_B}{2} |D| t
\end{equation*}
Hence, taking the supremum in time and expectations in \eqref{ulambda1}, for every $t \in [0, T]$, we have
\begin{equation*}
\begin{split}
& \frac{1}{2} \E \sup_{r \in [0, t]}\|u_{\lambda}(r)\|_H^{2} +\E \int_0^t \|\nabla u_{\lambda}(s)\|_H^2 ds \\
& \leq \frac{C_B}{2} |D| t +\frac{1}{2} \E\|u_0\|_H^2+ \frac{1}{2} \E \int_0^t  \|g(s)\|_H^2 ds  + \left(2c+\frac{1}{2}\right) \E \int_0^t \|u_{\lambda}(s)\|_H^2 ds \\
& + \E  \sup_{r \in [0, t]} \int_0^r (u_{\lambda}(s), B_{\lambda}(u_{\lambda}(s)))_H dW(s)
\end{split}
\end{equation*}
Finally, using Burkholder-Davis-Gundy inequality and Young's inequality in the stochastic integral in the last inequality, we have
\begin{align*}
& \E \sup_{r \in [0, t]} \int_0^t (u_{\lambda}(s), B_{\lambda}(u_{\lambda}(s)))_H dW(s) \\
&  \lesssim  \E \left(\int_0^t \|u_{\lambda}(s)\|_H^2 \|B_{\lambda}(u_{\lambda}(s))\|_{\mathcal{L}^2(U, H)}^2 ds\right)^{1/2} \\
 & \leq \E \left(\sup_{r \in [0, t]} \|u_{\lambda}(t)\|_H^2 \int_0^t \|B_{\lambda}(u_{\lambda}(s))\|_{\mathcal{L}^2(U, H)}^2 ds\right)^{1/2} \\
& \leq \frac{1}{4} \E\sup_{r \in [0, t]} \|u_{\lambda}(t)\|_H^2 + C_B |D|t
\end{align*}
Combining everything together leads, for every $t \in [0, T]$, to the following inequality
\begin{align*}
& \frac{1}{4} \E \sup_{r \in [0, t]}\|u_{\lambda}(r)\|_H^{2} +\E \int_0^t \|\nabla u_{\lambda}(s)\|_H^2 ds \\
& \lesssim \frac{3}{2} C_B |D|t + \frac{1}{2} \E\|u_0\|_H^2 + \frac{1}{2} \E \int_0^t  \|g(s)\|_H^2 ds  + \left(2c+\frac{1}{2}\right) \E \int_0^t \|u_{\lambda}(s)\|_H^2 ds \\
& \leq \frac{3}{2} C_B |D|t + \frac{1}{2} \E\|u_0\|_H^2 + \frac{1}{2} \E \int_0^t  \|g(s)\|_H^2 ds  + \left(2c+\frac{1}{2}\right) \E \int_0^t \sup_{r \in [0, s]}\|u_{\lambda}(r)\|_H^2 ds
\end{align*}
Applying Gronwall's lemma, we can conclude that there exists a constant $a$, independent of $\lambda$, such that
\begin{equation}\label{firstEstimate}
\|u_{\lambda}\|_{L^2(\Omega; C^{0}([0, T]; H))} + \|u_{\lambda}\|_{L^2(\Omega; L^2(0, T; V ))} \leq a
\end{equation}

\subsection{Second estimate}
Now we want to write Itô's formula for $\int_D F_{\lambda}(u_{\lambda})$. At this purpose we refer to the classical version of Itô's formula in \cite[Theorem 4.2]{Pardoux} and to appendix [\ref{generalizedIto}], obtaining that for every $t \in [0, T]$, $\mathbb{P}$-a.s.\ 
\begin{equation}\label{FLambda}
\begin{alignedat}{1}
& \int_D F_{\lambda}(u_{\lambda}(s)) + \int_0^t (F''_{\lambda}(u_{\lambda}(s)),|\nabla u_{\lambda}(s)|^2)_H ds + \int_0^t \|F'_{\lambda}(u_{\lambda}(s))\|_H^2 ds \\ 
& = \int_D F_{\lambda}(u_0) + \int_0^t (F'_{\lambda}(u_{\lambda}(s)), g(s))_H ds + \frac{1}{2} \int_0^t \sum_{k \in \mathbb{N}} (F''_{\lambda}(u_{\lambda}(s)), |B_{\lambda}(u_{\lambda}(s))u_k|^2)_H ds \\
& + \int_0^t (F'_{\lambda}(u_{\lambda}(s)), B_{\lambda}(u_{\lambda}(s)))_H dW(s)
\end{alignedat}
\end{equation}
Since $F''_{\lambda}(u_{\lambda}(s)) = \beta'_{\lambda}(u_{\lambda}(s)) - 2c$ and $\beta_{\lambda}$ is an increasing function, we have 
\begin{align*}
\int_0^t (F''_{\lambda}(u_{\lambda}(s)), |\nabla u_{\lambda}(s)|^2)_H ds & = \int_0^t (\beta'_{\lambda}(u_{\lambda}(s)), |\nabla u_{\lambda}(s)|^2)_H ds - 2c \int_0^t \|\nabla u_{\lambda}(s)\|_H^2 ds\\
&  \geq  - 2c \int_0^t  \|\nabla u_{\lambda}(s)\|_H^2 ds
\end{align*}
Using Hölder's inequality and Young's inequality for the integral containing the external force $g$, we obtain 
\begin{equation*}
\int_0^t (F'_{\lambda}(u_{\lambda}(s)), g(s))_H ds \leq \frac{1}{2}\int_0^t \|g(s)\|_H^2 ds+\frac{1}{2}\int_0^t \|F'_{\lambda}(u_{\lambda}(s))\|_H^2 ds
\end{equation*}
Moreover, since $\beta_{\lambda} (r) = \beta(J_{\lambda}(r))$ for any $r \in \mathbb{R}$, then it holds that $\beta'_{\lambda}(r) = \beta'(J_{\lambda}(r)) J'_{\lambda}(r)$, where the derivative of the function $\beta$ is given by 
\begin{equation*}
\beta'(r) = \frac{2}{1-r^2}, \quad r \in (-1,1)
\end{equation*}
Now, using the assumptions in $(A2)$ on $(h_k)_{k \in \mathbb{N}}$ and the fact that the resolvent $J_{\lambda}$ is non-expansive, we have 
\begin{align*}
& \frac{1}{2} \int_0^t \sum_{k \in \mathbb{N}} (F''_{\lambda}(u_{\lambda}(s)), |B_{\lambda}(u_{\lambda}(s))u_k|^2)_H ds \\
& \leq \frac{1}{2} \int_0^t \sum_{k \in \mathbb{N}} (\beta'_{\lambda}(u_{\lambda}(s)), |B_{\lambda}(u_{\lambda}(s))u_k|^2)_H ds \\
& = \frac{1}{2} \int_0^t \sum_{k \in \mathbb{N}} \int_D \beta'(J_{\lambda}(u_{\lambda}(s))) J'_{\lambda}(u_{\lambda}(s))|h_k(J_{\lambda}(u_{\lambda}(s)))|^2 ds\\
& \leq \int_0^t \sum_{k \in \mathbb{N}} \int_D \frac{1}{1-(J_{\lambda}(u_{\lambda}(s)))^2} |h_k(J_{\lambda} (u_{\lambda}(s)))-h_k(1)\|h_k(J_{\lambda}(u_{\lambda}(s)))-h_k(-1)| ds \\
& \leq C_B \int_0^t \int_D \frac{|J_{\lambda}(u_{\lambda}(s))-1\|J_{\lambda} (u_{\lambda}(s))+1|}{|J_{\lambda} (u_{\lambda}(s))-1\|J_{\lambda} (u_{\lambda}(s))+1|} ds  = C_B |D| t
\end{align*}
Therefore, since $F_{\lambda}(u_{\lambda}(t)) \geq 0$ and $F_{\lambda}(u_0) \leq F(u_0)$, taking the supremum in time and expectations in \eqref{FLambda}, it holds
\begin{align*}
\frac{1}{2} \E \int_0^T \|F'_{\lambda}(u_{\lambda}(s))\|_H^2 ds & \leq C_B |D| T + \E \int_D F(u_0) + \frac{1}{2} \E \int_0^T \|g(s)\|_H^2 ds + 2c \E \int_0^T \|\nabla u_{\lambda}(s)\|_H^2 ds \\
& +\E \sup_{t \in [0, T]} \int_0^t (F'_{\lambda}(u_{\lambda}(s)), B_{\lambda}(u_{\lambda}(s))) dW(s)
\end{align*}
Concerning the stochastic integral in the last inequality, proceeding in a similar way as in the estimate for $u_{\lambda}$, i.e.\ using the Burkholder-Davis-Gundy inequality and the Young's inequality, we obtain
\begin{align*}
& \E \sup_{t \in [0, T]} \int_0^t (F'_{\lambda}(u_{\lambda}(s)), B_{\lambda}(u_{\lambda}(s))) dW(s) \\
& \lesssim  \E \left(\int_0^T \|F'_{\lambda}(u_{\lambda}(s))\|_H^2 \|B_{\lambda}(u_{\lambda}(s))\|_{\mathcal{L}^2(U, H)}^2 ds\right)^{1/2} \\
& \leq  \E \left(\int_0^T \|F'_{\lambda}(u_{\lambda}(s))\|_H^2 ds\right)^{1/2} \left(\sup_{t \in [0, T]} \|B_{\lambda}(u_{\lambda}(t))\|_{\mathcal{L}^2(U; H)}^2\right)^{1/2} \\
& \leq \frac{1}{4} \E \int_0^T \|F'_{\lambda}(u_{\lambda}(s))\|_H^2 ds + \E\sup_{t \in [0, T]} \|B_{\lambda}(u_{\lambda}(t))\|_{\mathcal{L}^2(U; H)}^2 \\
& \leq  \frac{1}{4} \E \int_0^T \|F'_{\lambda}(u_{\lambda}(s))\|_H^2 ds +  C_B |D|
\end{align*}
Therefore, thanks to \eqref{firstEstimate}, we can conclude that there exists a constant $a$, independent of $\lambda$, such that
\begin{equation}\label{secondEstimate}
\|F'_{\lambda}(u_{\lambda})\|_{L^2(\Omega; L^2(0, T; H))} \leq a
\end{equation}
In particular, since $\beta_{\lambda}(u_{\lambda}) = F'_{\lambda}(u_{\lambda}) +2cu_{\lambda}$, from \eqref{firstEstimate} and \eqref{secondEstimate}, it follows that there exits a constant $a$ independent of $\lambda$ such that
\begin{equation}\label{estimateBetaLambda}
\|\beta_{\lambda}(u_{\lambda})\|_{L^2(\Omega; L^2(0, T, H))} \leq a
\end{equation}
\subsection{Third estimate}
Fix $\lambda$, $\epsilon \in (0,1)$ with $\lambda > \epsilon$ and consider equation \eqref{approxEquation} evaluated for $\lambda$ and $\epsilon$. Then, taking the difference, we obtain that, for any $v \in V$, it holds for every $t \in [0, T]$, $\mathbb{P}$-a.s.
\begin{align*}
\int_D (u_{\lambda}(t) - u_{\epsilon}(t))v &+ \int_0^t \int_D \nabla (u_{\lambda}(s) - u_{\epsilon}(s)) \nabla v ds + \int_0^t \int_D (F'_{\lambda}(u_{\lambda}(s))- F'_{\epsilon}(u_{\epsilon}(s))) v ds \\
& =\int_D \left(\int_0^t (B_{\lambda}(u_{\lambda}(s))-B_{\epsilon}(u_{\epsilon}(s))) dW(s)\right) v
\end{align*}
Therefore, applying Itô's formula for the square of the $H$-norm of $u_{\lambda}-u_{\epsilon}$, it holds for every $t \in [0, T]$, $\mathbb{P}$-a.s. 
\begin{equation}\label{estimateF'}
\begin{alignedat}{1}
& \frac{1}{2}\|u_{\lambda}(t)-u_{\epsilon}(t)\|_H^2 + \int_0^t \|\nabla (u_{\lambda}(s)-u_{\epsilon}(s))\|_H^2 ds + \int_0^t  (F'_{\lambda}(u_{\lambda}(s))-F'_{\epsilon}(u_{\epsilon}(s)),u_{\lambda}(s)-u_{\epsilon}(s))_H ds \\
& = \frac{1}{2}\int_0^t \|B_{\lambda}(u_{\lambda}(s))-B_{\epsilon}(u_{\epsilon}(s))\|_{\mathcal{L}^2(U, H)}^2 ds + \int_0^t (u_{\lambda}(s)-u_{\epsilon}(s), B_{\lambda}(u_{\lambda}(s))-B_{\epsilon}(u_{\epsilon}(s)))_H dW(s)
\end{alignedat}
\end{equation}
Now, since 
\begin{equation*}
F'_{\lambda}(u_{\lambda}(s))-F'_{\epsilon}(u_{\epsilon}(s)) = \beta_{\lambda}(u_{\lambda}(s))-\beta_{\epsilon}(u_{\epsilon}(s))-2c(u_{\lambda}(s)-u_{\epsilon}(s))
\end{equation*}
using the monotonicity of $\beta$, Hölder's inequality, and Young's inequality, we obtain
\begin{align*}
& (\beta_{\lambda}(u_{\lambda}(s))-\beta_{\epsilon}(u_{\epsilon}(s)), u_{\lambda}(s)-u_{\epsilon}(s))_H \\ 
& = (\beta_{\lambda}(u_{\lambda}(s))-\beta_{\epsilon}(u_{\epsilon}(s)), \lambda \beta_{\lambda}(u_{\lambda}(s)) -\epsilon \beta_{\epsilon}(u_{\epsilon}(s)))_H  \\
& \qquad + (\beta_{\lambda}(u_{\lambda}(s))-\beta_{\epsilon}(u_{\epsilon}(s)), J_{\lambda}(u_{\lambda}(s))-J_{\epsilon}(u_{\epsilon}(s)))_H \\
& \geq (\beta_{\lambda}(u_{\lambda}(s))-\beta_{\epsilon}(u_{\epsilon}(s)), \lambda \beta_{\lambda}(u_{\lambda}(s)) -\epsilon \beta_{\epsilon}(u_{\epsilon}(s)))_H \\
& = \lambda \|\beta_{\lambda}(u_{\lambda}(s)\|_H^2 + \epsilon \|\beta_{\epsilon}(u_{\epsilon}(s)\|_H^2-(\lambda+\epsilon)(\beta_{\lambda}(u_{\lambda}(s)), \beta_{\epsilon}(u_{\epsilon}(s)))_H \\
& \geq -(\lambda+\epsilon)\|\beta_{\lambda}(u_{\lambda}(s))\|_H \|\beta_{\epsilon}(u_{\epsilon}(s))\|_H \\
& \geq -\frac{\lambda+\epsilon}{2} \left(\|\beta_{\lambda}(u_{\lambda}(s))\|_H^2 +\|\beta_{\epsilon}(u_{\epsilon}(s))\|_H^2\right) 
\end{align*}
Using the assumptions in $(A2)$ on $(h_k)_{k \in \mathbb{N}}$, we obtain
\begin{align*}
\frac{1}{2} \E \int_0^t \|B_{\lambda}(u_{\lambda}(s))-B_{\epsilon}(u_{\epsilon}(s))\|_{\mathcal{L}^2(U, H)}^2 & = \frac{1}{2} \E \int_0^t \sum_{k \in \mathbb{N}} \int_D |h_k(J_{\lambda} (u_{\lambda}(s)))-h_k(J_{\epsilon}(u_{\epsilon}(s)))|^2 ds \\
& \leq \frac{C_B}{2} \E \int_0^t \|J_{\lambda}(u_{\lambda}(s))-J_{\epsilon} (u_{\epsilon}(s))\|_H^2 ds
\end{align*}
Hence, taking the supremum in time and expectations in \eqref{estimateF'}, for every $t \in [0, T]$, we have
\begin{align*}
& \frac{1}{2} \E \sup_{r \in [0, t]} \|u_{\lambda}(r)-u_{\epsilon}(r)\|_H^2+ \E \int_0^t \|\nabla(u_{\lambda}(s)-u_{\epsilon}(s))\|_H^2 ds \\
& \leq \frac{\lambda+\epsilon}{2} \E \int_0^t \left(\|\beta_{\lambda}(u_{\lambda}(s))\|_H^2 +\|\beta_{\epsilon}(u_{\epsilon}(s))\|_H^2\right) ds + 2c \E \int_0^t \|u_{\lambda}(s)-u_{\epsilon}(s)\|_H^2 ds \\
& + \frac{C_B}{2} \E \int_0^t \|J_{\lambda}(u_{\lambda}(s))-J_{\epsilon} (u_{\epsilon}(s))\|_H^2 ds + \E \sup_{r \in [0, t]} \int_0^r (u_{\lambda}(s)-u_{\epsilon}(s), B_{\lambda}(u_{\lambda}(s))-B_{\epsilon}(u_{\epsilon}(s))) dW(s)
\end{align*}
Thanks to \eqref{estimateBetaLambda} it exists a constant $a_1$ independent of $\lambda$ and $\epsilon$, such that 
\begin{equation*}
\E \int_0^t \left(\|\beta_{\lambda}(u_{\lambda}(s))\|_H^2 +\|\beta_{\epsilon}(u_{\epsilon}(s))\|_H^2\right) ds  \leq 2 a_1
\end{equation*}
Using Burkholder-Davis-Gundy inequality and the Young's inequality we have 
\begin{align*}
& \E \sup_{r \in [0, t]} \int_0^r (u_{\lambda}(s)-u_{\epsilon}(s), B_{\lambda}(u_{\lambda}(s))-B_{\epsilon}(u_{\epsilon}(s))) dW(s) \\
& \lesssim \E \left(\int_0^t \|u_{\lambda}(s)-u_{\epsilon}(s)\|_H^2 \|B_{\lambda}(u_{\lambda}(s))-B_{\epsilon}(u_{\epsilon}(s))\|_{\mathcal{L}^2(U, H)}^2 ds\right)^{1/2} \\
& \leq \frac{1}{4} \E \sup_{r \in [0, t]} \int_0^r \|u_{\lambda}(s)-u_{\epsilon}(s)\|_H^2 ds + C_B \E \int_0^t \|J_{\lambda} (u_{\lambda}(s))-J_{\epsilon}(u_{\epsilon}(s))\|_H^2 ds
\end{align*}
Since for every $r \in \mathbb{R}$ the map $\lambda \mapsto |\beta_{\lambda}(r)|$ is non-decreasing for $\lambda \searrow 0$ and since we chose $\lambda > \epsilon$, we have 
\begin{align*}
\|J_{\lambda}(u_{\lambda}(s))-J_{\epsilon}(u_{\epsilon}(s))\|_H^2 & \leq 2 \|J_{\lambda}(u_{\lambda}(s))-J_{\lambda}(u_{\epsilon}(s))\|_H^2+2\|J_{\lambda}(u_{\epsilon}(s))-J_{\epsilon}(u_{\epsilon}(s))\|_H^2 \\
& \leq 2 \|u_{\lambda}(s)-u_{\epsilon}(s)\|_H^2 + 2 \|J_{\lambda}(u_{\epsilon}(s))-u_{\epsilon}(s) - (J_{\epsilon}(u_{\epsilon}(s))-u_{\epsilon}(s))\|_H^2 \\
& \leq 2 \|u_{\lambda}(s)-u_{\epsilon}(s)\|_H^2 + 4(\lambda^2 \|\beta_{\lambda}(u_{\epsilon}(s))\|_H^2+\epsilon^2 \|\beta_{\epsilon}(u_{\epsilon}(s))\|_H^2) \\
& \leq 2 \|u_{\lambda}(s)-u_{\epsilon}(s)\|_H^2 + 4(\lambda^2+\epsilon^2) \|\beta_{\epsilon}(u_{\epsilon}(s))\|_H^2\\
\end{align*}
Hence, thanks to \eqref{estimateBetaLambda}, it exists a constant $a_2$ independent of $\lambda$ and $\epsilon$, such that 
\begin{equation*}
\E \int_0^t \|J_{\lambda}(u_{\lambda}(s))-J_{\epsilon}(u_{\epsilon}(s))\|_H^2 ds \leq 2 \E \int_0^t \|u_{\lambda}(s)-u_{\epsilon}(s)\|_H^2 + 2 a_2 (\lambda^2+\epsilon^2)
\end{equation*}
Combining everything together, we obtain that for every $t \in [0, T]$, it holds
\begin{align*}
& \frac{1}{4} \E \sup_{r \in [0, t]} \|u_{\lambda}(r)-u_{\epsilon}(r)\|_H^2+ \E \int_0^t \|\nabla(u_{\lambda}(s)-u_{\epsilon}(s))\|_H^2 ds  \\
& \lesssim [a_1(\lambda+\epsilon) +3 C_B a_2 (\lambda^2+\epsilon^2)] +(2c + 3 C_B) \E \int_0^t \|u_{\lambda}(s)-u_{\epsilon}(s)\|_H^2 ds \\
& \leq [a_1(\lambda+\epsilon) +3 C_B a_2 (\lambda^2+\epsilon^2)]  +(2c + 3 C_B) \E \int_0^t  \sup_{r \in [0, s]}\|u_{\lambda}(r)-u_{\epsilon}(r)\|_H^2 ds
\end{align*}
Finally, thanks to Gronwall's lemma, for every $t \in [0, T]$ it holds 
\begin{equation*}
\E \sup_{r \in [0, t]} \|u_{\lambda}(r)-u_{\epsilon}(r)\|_H^2+ \E \int_0^t \|\nabla(u_{\lambda}(s)-u_{\epsilon}(s))\|_H^2 ds  \lesssim [a_1(\lambda+\epsilon) +3 C_B a_2 (\lambda^2+\epsilon^2)]e^{(2c+3C_B) t}
\end{equation*}
and in particular we have
\begin{equation}\label{thirdEstimate}
\|u_{\lambda}-u_{\epsilon}\|_{L^2(\Omega; C^{0}([0, T]; H))} + \|u_{\lambda}-u_{\epsilon}\|_{L^2(\Omega; L^2(0, T; V ))} \to 0, \quad \text{ as } \lambda, \epsilon \to 0^{+}
\end{equation}

\section{Proofs of main results}\label{proofMainResult}
\subsection{Existence of a variational solution}
Now we are concerned with passing to the limit as $\lambda \to 0^{+}$ in equation \eqref{approxEquation}. We use the uniform estimate with respect to $\lambda$ in order to deduce the existence of a limit process $u$, which is a variational solution in the sense of the Definition \ref{varSol}, to problem \eqref{acSpde}.

\smallskip
From \eqref{thirdEstimate} it follows that $(u_{\lambda})_{\lambda \in (0,1)}$ is a Cauchy sequence in $L^2(\Omega; C^0([0, T]; H)) \cap L^2(\Omega; L^2(0, T; V))$ and therefore 
\begin{equation}\label{convergenceU}
u_{\lambda} \to u \quad \text{ in } L^2(\Omega; C^0([0, T]; H)) \cap L^2(\Omega; L^2(0, T; V))
\end{equation}
Moreover, since $-\Delta \in \mathcal{L}(V, V^{*})$, we have  
\begin{equation}\label{convergenceLap}
-\Delta u_{\lambda} \to -\Delta u \quad \text{ in } L^2(\Omega; L^2(0, T; V^{*}))
\end{equation}
We recall that if $E$ is a reflexive Banach space and $p \in (1, \infty) $, then $L^p(\Omega; E)$ is reflexive and its dual is $L^{p/(p-1)}(\Omega, E^{*})$. Hence $L^2(\Omega; L^2(0, T; H))$ is a reflexive Banach space and so, thanks to \eqref{secondEstimate}, there exists a subsequence $\lambda_n \to 0^{+}$, as $n \to \infty$, such that
\begin{equation}\label{convergenceF'}
F'_{\lambda_n}(u_{\lambda_n}) \rightharpoonup \xi \quad \text{ in } L^2(\Omega; L^2(0, T; H))
\end{equation} 
Since $F'_{\lambda}(u_{\lambda}) = \beta_{\lambda}(u_{\lambda})-2cu_{\lambda}$, from \eqref{convergenceU} and \eqref{convergenceF'}, as $n \to \infty$, we obtain  
\begin{equation*}
\beta_{\lambda_n}(u_{\lambda_n}) \rightharpoonup \xi + 2 cu \quad \text{ in } L^2(\Omega; L^2(0, T; H))
\end{equation*}
Moreover, thanks to \eqref{estimateBetaLambda}, we have 
\begin{align*}
\|J_{\lambda}(u_{\lambda}) - u\|_{L^2{(\Omega; L^2(0, T; H))}} & \leq \|J_{\lambda}(u_{\lambda}) - u_{\lambda}\|_{L^2(\Omega; L^2(0, T; H)} + \|u-u_{\lambda}\|_{L^2(\Omega; L^2(0, T; H))} \\
& \leq \lambda \|\beta_{\lambda}(u_{\lambda})\|_{L^2(\Omega; L^2(0, T; H))} + \|u_{\lambda}-u\|_{L^2(\Omega; L^2(0, T; H))} \\
& \leq a \lambda + \|u_{\lambda}-u\|_{L^2(\Omega; L^2(0, T; H))} \to 0 \quad \text{ as } \lambda \to 0^{+} 
\end{align*}
Since $\beta_{\lambda}(r) = \beta(J_{\lambda}(r))$ for any $r \in \mathbb{R}$, we have the following convergences as $n \to \infty$
\begin{alignat*}{2}
& J_{\lambda_n} u_{\lambda_n} \to  u && \quad \text{ in } L^2(\Omega; L^2(0, T; H)) \\
& \beta(J_{\lambda_n} u_{\lambda_n}) \rightharpoonup \xi + 2 c u && \quad \text{ in } L^2(\Omega; L^2(0, T; H))
\end{alignat*}
Since $\beta \subset \mathbb{R} \times \mathbb{R}$ is a maximal monotone graph, then it is strongly-weakly closed \cite[Proposition 2.1]{Barbu}, and so we have
\begin{equation*}
\xi = \beta(u)-2cu = F'(u) \quad \text{ a.e.\ in } [0, T] \times \Omega
\end{equation*}
Therefore, as $n \to + \infty$, we have 
\begin{equation}\label{convergenceF'real}
F'_{\lambda_n}(u_{\lambda_n}) \rightharpoonup F'(u) \quad \text{ in } L^2(\Omega; L^2(0, T; H))
\end{equation}
We are now concerned with the passage to the limit in the stochastic integral. We recall that the function $F'$ is only defined in the interval $(-1, 1)$ and since $F'(u) \in L^2(\Omega; L^2(0, T; H))$, then we have
\begin{equation*}
||u||_{L^{\infty}(D)} < 1 \quad \text{ a.e.\ in } [0, T] \times \Omega
\end{equation*}
Therefore, since the operator $B$ is only defined in $H_1$, we can evaluate $B$ at $u$, and using Burkholder-Davis-Gundy inequality, we obtain
\begin{align*}
\E \sup_{t \in [0, T]} \left\| \int_0^ t (B_{\lambda}(u_{\lambda}(s))  - B (u(s)))d W(s) \right\|_H^2 & \lesssim \E \int_0^T \|B(J_{\lambda}(u_{\lambda}(s))) - B(u(s))\|_{\mathcal{L}^2(U, H)}^2 ds \\
& \leq C_B \|J_{\lambda}(u_{\lambda}) - u\|^2_{L^2(\Omega; L^2(0, T; H))} \to 0 \quad \text{ as } \lambda \to 0^{+} 
\end{align*}
and therefore, as $\lambda \to 0^{+}$, it holds
\begin{equation}\label{convergenceStoch}
B_{\lambda}(u_{\lambda}) \cdot W \to B(u) \cdot W \quad \text{ in } L^2(\Omega; C^0([0, T]; H))
\end{equation}
We are now ready to pass to the limit in equation \eqref{approxEquation}. Thanks to \eqref{convergenceU}, \eqref{convergenceLap}, \eqref{convergenceF'real}, and \eqref{convergenceStoch}, for any $v \in V$ and for every $t \in [0, T]$, as $n \to \infty$, we have
\begin{align*}
& \E \,  \sup_{t \in [0, T]} \left| \int_D u_{\lambda_n}(t) v - \int_D u(t) v \right| \to 0 \\
& \E \int_0^t \int_D \nabla u_{\lambda_n}(s) \nabla v ds \to \E \int_0^t \int_D \nabla u(s) \nabla v ds \\
& \E \int_0^t \int_D F'_{\lambda_n}(u_{\lambda_n}(s)) v ds \to \E \int_0^t \int_D F'(u(s)) v ds\\
& \E \sup_{t \in [0, T]} \left|\int_D \left(\int_0^t B_{\lambda_n}(u_{\lambda_n}(s)) dW(s)\right) v -\int_D \left(\int_0^t B(u(s)) dW(s)\right) v\right| \to 0
\end{align*}
At this point, evaluating \eqref{approxEquation} in $\lambda_n$ and letting $n \to \infty$, we obtain
\begin{align*}
\int_D u(t) v  & +\int_0^t \int_D \nabla u(s) \nabla v ds +\int_0^t \int_D F'(u(s) v ds \\
& =  \int_D u_0 v + \int_0^t \int_D g(s) v ds +  \int_D \left(\int_0^t B(u(s)) dW(s)\right) v
\end{align*}
for every $t \in [0, T]$, $\mathbb{P}$-a.s. Therefore the limit process $u$ is a variational solution to problem \eqref{acSpde}, indeed $u \in L^2(\Omega; C^{0}([0, T]; H)) \cap L^2(\Omega; L^2(0, T; V))$, it is such that $F'(u) \in L^2(\Omega; L^2(0, T; H))$ and for all $v \in V$ it satisfies \eqref{varSol} for every $t \in [0, T]$, $\mathbb{P}$-a.s.

\subsection{Continuous dependence on the initial datum}
Assume that $F$, $B$, $(u_0^1, g_1)$, $(u_0^2, g_2)$ are as in assumptions $(A1)$-$(A3)$ and let $u_1$ and $u_2$ the variational solutions of the respective problems. Then, applying Itô's formula for the square of the $H$-norm of $u_1-u_2$, we obtain for every $t \in [0, T]$, $\mathbb{P}$-a.s.
\begin{equation}\label{itoContDep}
\begin{alignedat}{1}
& \frac{1}{2} \|u_1(t)-u_2(t)\|_H^2 +\int_0^t \|\nabla(u_1(s)-u_2(s))\|_H^2 ds + \int_0^t (F'(u_1(s))-F'(u_2(s)), u_1(s)-u_2(s))_H ds  \\
& = \frac{1}{2} \|u_0^1-u_0^2\|_H^2+\int_0^t (g_1(s)-g_2(s), u_1(s)-u_2(s))_H ds +\frac{1}{2}\int_0^t \|B(u_1(s))-B(u_2(s))\|_{\mathcal{L}^2(U, H)}^2 ds \\
& + \int_0^t (u_1(s)-u_2(s), B(u_1(s))-B(u_2(s)))_H dW(s)
\end{alignedat}
\end{equation}
Since $F'(r) = \beta(r)-2cr$ for any $r \in (-1,1)$, and $\beta$ is a maximal monotone graph, we have  
\begin{align*}
& \int_0^t (F'(u_1(s))-F'(u_2(s)), u_1(s)-u_2(s))_H ds \\
& =   \int_0^t (\beta(u_1(s))-\beta(u_2(s)), u_1(s)-u_2(s))_H ds-2c \int_0^t \|u_1(s)-u_2(s)\|_H^2 ds \\
& \geq - 2c \int_0^t \|u_1(s)-u_2(s)\|_H^2 ds 
\end{align*}
Using Young's inequality we obtain
\begin{equation*}
\int_0^t (g_1(s)-g_2(s), u_1(s)-u_2(s))_H ds \leq \frac{1}{2} \int_0^t \|g_1(s)-g_2(s)\|_{V^{*}}^2 ds + \frac{1}{2} \int_0^t \|u_1(s)-u_2(s)\|_V^2 ds
\end{equation*}
Using the Lipschitz-continuity of the operator $B$ we have   
\begin{equation*}
\frac{1}{2} \int_0^t \|B(u_1(s))-B(u_2(s))\|_{\mathcal{L}^2(U, H)}^2 ds \leq \frac{C_B}{2} \int_0^t \|u_1(s)-u_2(s)\|_H^2 ds
\end{equation*}
Hence, taking the supremum in time and expectations in \eqref{itoContDep}, for every $t \in [0, T]$, we have
\begin{align*}
&\frac{1}{2} \E \sup_{r \in [0, t]} \|u_1(r)-u_2(r)\|_H^2 + \frac{1}{2} \E \int_0^t \|\nabla(u_1(s)-u_2(s))\|_H^2 ds \\
& \leq \frac{1}{2} \E \|u_0^1-u_0^2\|_H^2 +\frac{1}{2} \E \int_0^t \|g_1(s)-g_2(s)\|_{V^{*}}^2 + \left(\frac{1}{2} + 2c + \frac{C_B}{2}\right) \E \int_0^t \|u_1(s)-u_2(s)\|_H^2 ds\\
& + \E \sup_{r \in [0, t]} \int_0^r (u_1(s)-u_2(s), B(u_1(s))-B(u_2(s)))_H dW(s)
\end{align*}
Now, using Burkholder-Davis-Gundy inequality and Young's inequality in the stochastic integral in the last inequality, we obtain
\begin{align*}
& \E \sup_{r \in [0, t]} \int_0^r (u_1(s)-u_2(s), B(u_1(s))-B(u_2(s)))_H dW(s) \\
& \lesssim\E\left(\int_0^t \|u_1(s)-u_2(s)\|_H^2 \|B(u_1(s))-B(u_2(s))\|_{\mathcal{L}^2(U,H)}^2 ds\right)^{\frac{1}{2}} \\
& \leq \frac{1}{4} \sup_{r \in [0, t]} \|u_{1}(r)-u_2(r)\|_H^2 + C_B \E \int_0^t \|u_1(s)-u_2(s)\|_H^2 ds
\end{align*}
and so, for every $t \in [0, T]$ we have
\begin{align*}
&\frac{1}{4} \E \sup_{r \in [0, t]} \|u_1(r)-u_2(r)\|_H^2 + \frac{1}{2}\E \int_0^t \|\nabla(u_1(s)-u_2(s))\|_H^2 ds \\
& \lesssim \frac{1}{2} \E \|u_0^1-u_0^2\|_H^2 +\frac{1}{2} \E \int_0^t \|g_1(s)-g_2(s)\|_{V^{*}}^2 + \left(\frac{1}{2} + 2c + \frac{3}{2} C_B\right) \E \int_0^t \|u_1(s)-u_2(s)\|_H^2 ds \\
& \leq \frac{1}{2} \E \|u_0^1-u_0^2\|_H^2 +\frac{1}{2} \E \int_0^t \|g_1(s)-g_2(s)\|_{V^{*}}^2 + \left(\frac{1}{2} + 2c + \frac{3}{2} C_B\right) \E \int_0^t \sup_{r \in [0, s]} \|u_1(r)-u_2(r)\|_H^2 ds 
\end{align*}
Therefore, applying Gronwall's lemma, we obtain the desired result \eqref{continuousDep}.

\subsection{Existence of an analytically strong solution}
Assuming that $u_0 \in L^2(\Omega; V)$, we now show an additional estimate on the approximated solution $u_{\lambda}$, from which it follows that the solution $u$ is also an analytically strong solution to problem \eqref{acSpde}, in the sense of Definition \ref{analyticalStrongSol}.

\smallskip
Consider the functional $\Phi$ defined as follows
\begin{equation*}
\Phi: V \to \mathbb{R}, \quad \Phi(x) := \frac{1}{2}\|\nabla x\|_H^2, \quad x \in V
\end{equation*}
Its first order Fréchet derivative is given by $D \Phi(x) = -\Delta x$ and, since $-\Delta \in \mathcal{L}(V, V^{*})$, its second order Fréchet derivative is given by $D^2 \Phi(x) = - \Delta$. For fixed $\lambda \in (0, 1)$, using a finite-dimensional approximation for the approximated problem \eqref{approximatedProblem}, thanks to the fact that $u_0 \in L^2(\Omega; V)$, it is possible to prove that $u_{\lambda} \in L^2(\Omega; L^2(0, T; V_0))$. However, for brevity we omit the details and proceed in a formal way. Therefore, thanks to \cite[Theorem 4.32]{DaPrato}, Itô's formula for $\Phi(u_{\lambda})$ reads for every $t \in [0, T]$, $\mathbb{P}$-a.s.
\begin{equation}\label{estimateGrad}
\begin{alignedat}{1}
& \frac{1}{2} \|\nabla u_{\lambda}(t)\|_H^2+\int_0^t \|\Delta u_{\lambda}(s)\|_H^2+\int_0^t  (F''_{\lambda}(u_{\lambda}(s)), |\nabla u_{\lambda}(s)|^2)_H ds \\
& = \frac{1}{2}\|\nabla u_0\|_H^2 - \int_0^t  (\Delta u_{\lambda}(s), g(s))_H ds + \frac{1}{2}\int_0^t \|\nabla B_{\lambda}(u_{\lambda}(s))\|_{\mathcal{L}^2(U, H)}^2 ds \\
& +\int_0^t (\nabla u_{\lambda}(s), \nabla B_{\lambda}(u_{\lambda}(s)))_H dW(s)
\end{alignedat}
\end{equation}
Since $F''_{\lambda}(u(s)) = \beta'_{\lambda}(u_{\lambda}(s))-2c$ and $\beta_{\lambda}$ is an increasing function, we have 
\begin{align*}
\int_0^t  (F''_{\lambda}(u_{\lambda}(s)), |\nabla u_{\lambda}(s)|^2)_H ds & = \int_0^t  (\beta'_{\lambda}(u_{\lambda}(s)), |\nabla u_{\lambda}(s)|^2)_H ds - 2c \int_0^t \|\nabla u_{\lambda}(s)\|_H^2 ds\\ 
&  \geq - 2c \int_0^t \|\nabla u_{\lambda}(s)\|_H^2 ds
\end{align*}
Using Hölder's inequality and Young's inequality for the integral containing the external force $g$, we obtain 
\begin{equation*}
-\int_0^t (\Delta(u_{\lambda}(s)), g(s))_H ds \leq \frac{1}{2}\int_0^t \|g(s)\|_H^2 ds+\frac{1}{2}\int_0^t \|\Delta(u_{\lambda}(s))\|_H^2 ds
\end{equation*}
Moreover, using the assumptions in $(A2)$ on $(h_k)_{k \in \mathbb{N}}$, we have 
\begin{align*}
\frac{1}{2}\int_0^t \|\nabla B_{\lambda}(u_{\lambda}(s))\|_{\mathcal{L}^2(U, H)}^2 ds & = \frac{1}{2}\int_0^t \sum_{k \in \mathbb{N}} \int_D |h'_k(u_{\lambda}(s)) \nabla u_{\lambda}(s)|^2 ds \\
& \leq \frac{C_B}{2}\int_D \|\nabla u_{\lambda}(s)\|_H^2 ds
\end{align*}
Hence, taking the supremum in time and the expectation in \eqref{estimateGrad}, for every $t \in [0, T]$ it holds
\begin{align*}
& \frac{1}{2} \E \sup_{r \in [0, t]} \|\nabla u_{\lambda}(t)\|_H^2+\frac{1}{2}\E \int_0^t \|\Delta u_{\lambda}(s)\|_H^2 \\
& \leq \frac{1}{2} \E \|\nabla u_0\|_H^2 + \frac{1}{2} \E \int_0^t \|g(s)\|_H^2 ds + \left(2c+\frac{C_B}{2}\right) \E \int_0^t \|\nabla u_{\lambda}(s)\|_H^2 ds\\
&  + \E \sup_{r \in [0, t]} \int_0^r (\nabla u_{\lambda}(s), \nabla B_{\lambda}(u_{\lambda}(s)))_H dW(s)
\end{align*}
Finally, using Burkholder-Davis-Gundy inequality, Hölder's inequality, and Young's inequality on the stochastic integral in the last inequality, we obtain 
\begin{align*}
& \E \sup_{r \in [0, t]} \int_0^t (\nabla u_{\lambda}(s), \nabla B_{\lambda}(u_{\lambda}(s)))_H dW(s) \\
& \lesssim \E \left(\int_0^t\|\nabla u_{\lambda}(s)\|_H^2 \|\nabla B_{\lambda}(u_{\lambda}(s))\|_{\mathcal{L}^2(U, H)}^2ds\right)^{1/2} \\
& \leq \E \left(\sup_{r \in [0, t]} \|\nabla u_{\lambda}(t)\|_H^2 \right)^{1/2} \left(\int_0^t \|\nabla B_{\lambda}(u_{\lambda}(s))\|_{\mathcal{L}^2(U, H)}^2 ds \right)^{1/2} \\
& \leq \frac{1}{4} \E \sup_{r \in [0, t]} \|\nabla u_{\lambda}(t)\|_H^2  + C_B \E \int_0^t \|\nabla u_{\lambda}(s)\|_H^2 ds
\end{align*}
and so, for every $t \in [0, T]$ we have 
\begin{align*}
& \frac{1}{4} \E \sup_{r \in [0, t]} \|\nabla u_{\lambda}(t)\|_H^2+\frac{1}{2} \E \int_0^t \|\Delta u_{\lambda}(s)\|_H^2\\
&  \lesssim \frac{1}{2} \E \|\nabla u_0\|_H^2 + \frac{1}{2} \E \int_0^t \|g(s)\|_H^2 ds  + \left(2c+\frac{3}{2}C_B\right) \E \int_0^t \|\nabla u_{\lambda}(s)\|_H^2 ds\\
& \leq \frac{1}{2} \E \|\nabla u_0\|_H^2 + \frac{1}{2} \E \int_0^t \|g(s)\|_H^2 ds  + \left(2c+\frac{3}{2}C_B\right) \E \int_0^t  \sup_{r \in [0, s]}\|\nabla u_{\lambda}(r)\|_H^2 ds
\end{align*}
Applying Gronwall's lemma, we can conclude that there exits a constant $a$, independent of $\lambda$, such that
\begin{equation}\label{fourthEstimate}
\|u_{\lambda}\|_{L^2(\Omega; L^{\infty}(0, T; V))} + \|u_{\lambda}\|_{L^2(\Omega; L^2(0, T; V_0))} \leq a 
\end{equation}
Therefore, thanks to \eqref{convergenceU} and \eqref{fourthEstimate}, using the lower semicontinuity of the norms, we have that the the limit process $u \in L^2(\Omega; C^{0}([0, T]; H)) \cap L^2(\Omega; L^2(0, T; V_0)) \cap L^2(\Omega; L^{\infty}(0, T; V))$ is an analytically strong solution to problem \eqref{acSpde}.

\subsection{Estimates on the derivatives}
Fix $n \geq 2$ and assume $(A1)$, $(A2)_n$, and $(A3)_n$. Let $u$ be the variational solution to $\eqref{acSpde}$ and consider the function $G_n$ as defined in \eqref{Gn}. We want to write Itô's formula for $\int_D G_n(u)$. However, since $G'_n$ is not Lipschitz-continuous and $G''_n$ is not bounded, we cannot immediately apply the classical result \cite[Theorem 4.2]{Pardoux}, as we have done in Appendix [\ref{generalizedIto}]. A rigorous approach would require an approximation on the function $G_n$ in order to recover such regularity. For example, since $G'_n$ is monotone increasing and continuous, then we can identify it with a maximal monotone graph in $\mathbb{R} \times \mathbb{R}$. Therefore, for every $\lambda \in (0, 1)$, we can consider its Yosida approximation $G'_{n, \lambda}$ and define
\begin{equation*}
G_{n, \lambda}(r):= \int_0^r G'_{n, \lambda}(s) ds, \quad r \in \mathbb{R}
\end{equation*}
In this way we have that $G_{n, \lambda} \in C^2(\mathbb{R})$ is such that $G'_{n, \lambda}$ is Lipschitz-continuous and $G''_{n, \lambda}$ is continuous and bounded, and so we can apply \cite[Theorem 4.2]{Pardoux} to $\int_D G_{n, \lambda}(u)$. However, since this is not restrictive in our direction, we shall proceed formally in order to make the treatment lighter and to avoid heavy notations. Then, for every $t \in [0, T]$, $\mathbb{P}$-a.s.\
\begin{equation}\label{itoDer}
\begin{alignedat}{1}
& \int_D G_n(u(t)) + \int_0^t  (G''_n(u(s)), |\nabla u(s)|^2)_H ds + \int_0^t (G'_n(u(s)), F'(u(s)))_H ds\\
& = \int_D G_n(u_0) + \int_0^t (G'_n(u(s)), g(s))_H ds + \frac{1}{2} \int_0^t \sum_{k \in \mathbb{N}} \int_D G''_n(u(s)) |h_k(u(s))|^2 ds \\
& + \int_0^t (G'_n(u(s)), B(u(s)))_H d W(s)
\end{alignedat}
\end{equation}
Since for $r \in (-1, 1)$ we have that $F'(r) = \beta(r)-2cr$, then there exist $\underline{r}$, $\overline{r} \in (-1,1)$ such that for $r \in (-1, \underline{r})$ it holds that $F'(r) \leq -2$, while for $r \in (\overline{r}, 1)$ it holds that $F'(r) \geq 2$. Hence 
\begin{align*}
& \int_0^t (G'_n(u(s)), F'(u(s)))_H ds  = \int_0^t \int_{\{u(s) \in (-1, \underline{r})\}} G'_n(u(s)) F'(u(s)) ds\\
& + \int_0^t \int_{\{u(s) \in [\underline{r}, \overline{r}]\}} G'_n(u(s)) F'(u(s)) ds  + \int_0^t \int_{\{u(s) \in (\overline{r}, 1)\}} G'_n(u(s))F'(u(s))ds\\
& \geq 2 \int_0^t \int_{\{u(s) \in (-1, \underline{r}) \cup (\overline{r}, 1)\}} |G'_n(u(s))|ds + \int_0^t \int_{\{u(s) \in [\underline{r}, \overline{r}]\}} G'_n(u(s)) F'(u(s)) ds
\end{align*}
Thanks to the assumptions on the external force $g$ in $(A3)_n$, we also have  
\begin{equation}
\int_0^t (G'_n(u(s)), g(s))_H ds \leq \int_0^t \int_D |G'_n(u(s))| ds
\end{equation}
Moreover, using the assumptions in $(A2)_n$ on $(h_k)_{k \in \mathbb{N}}$, for every $k \in \mathbb{N}$ and for a.e.\ $(s, \omega)  \in (0,t)\times\Omega$, we have
\begin{align*}
& h_k(u(s)) = \frac{(u(s)-1)^{n+1}}{n!}\int_0^1 h_k^{(n+1)}(1+r(u(s)-1)) (1-r)^n dr\\
& h_k(u(s)) = \frac{(u(s)+1)^{n+1}}{n!}\int_0^1 h_k^{(n+1)}(-1+r(u(s)+1)) (1-r)^n dr
\end{align*}
which implies that
\begin{equation} \label{estimateHk}
\begin{alignedat}{1}
& |h_k(u(s))| \leq \frac{||h_k||_{W^{n+1, \infty}(-1, 1)}}{(n+1)!} (u(s)-1)^{n+1}\\
& |h_k(u(s))| \leq \frac{||h_k||_{W^{n+1, \infty}(-1, 1)}}{(n+1)!} (u(s)+1)^{n+1}\\
\end{alignedat}
\end{equation}
Therefore, thanks to \eqref{condDerGn}, \eqref{Cn}, and \eqref{estimateHk} we have 
\begin{align*}
& \frac{1}{2} \int_0^t  \sum_{k \in \mathbb{N}} \int_D G''_n(u(s)) |h_k(u(s))|^2 ds  \\
& \leq \frac{a''_{n}}{2}\int_0^t \sum_{k \in \mathbb{N}} \int_D \frac{|h_k(u(s))|^2}{(1-u(s)^2)^{n+1}} ds\\
& \leq \frac{a''_{n}}{2 (n+1)!^2} \int_0^t \sum_{k \in \mathbb{N}} ||h_k||^2_{W^{n+1, \infty}(-1, 1)}  \int_D \left( \frac{|u(s)-1| |u(s)+1|}{|u(s)-1| |u(s)+1|} \right)^{n+1} ds \\
& = \frac{a''_{n}}{2 (n+1)!^2} C_n t |D|
\end{align*}
Now, taking expectations in \eqref{itoDer}, combining all the information together, using the fact that $G''_n(r) \geq 0$ for all $r \in (-1,1)$, and rearranging the terms, we obtain the following inequality
\begin{align*}
& \E \int_D G_n(u(t))+ \E \int_0^t \int_D |G'_n (u(s))| ds\\
& \leq \frac{a''_{n}}{2 (n+1)!^2} C_n t |D| + \E \int_D G_n(u_0) + \E \int_0^t \int_{\{u(s) \in [\underline{r}, \overline{r}]\}} \left(|G'_n(u(s))| - G'_n(u(s))F'(u(s))\right) ds 
\end{align*}
Thanks to the continuity of the functions $G_n'$ and $F'$, the last term on the right hand-side in the last inequality is bounded by a constant $a^1_n$. Therefore, taking the supremum in time, we obtain  
\begin{equation*}
\sup_{t \in [0, T]} \E \int_D G_n (u(t))+\E \int_0^T \int_D |G'_n(u(s))| ds \leq a^1_n+ \frac{a''_{n}}{2 (n+1)!^2} C_n T |D|+\E \int_D G_n(u_0) 
\end{equation*}
hence, there exists a constant $a_n$ such that 
\begin{equation*}
\|G_n(u)\|_{L^{\infty}(0, T; L^1(\Omega; D))} + \|G'_n(u)\|_{L^1(\Omega; L^1(0, T; D)} \leq a_n
\end{equation*}
Now, since estimating $|G'_n|$ coincides with estimating $G_{n+1}$ and estimating $G_n$ is equivalent to estimate $|F^{(n)}|$, the result \eqref{estimateDerLog} is proved.

\appendix 
\section{A generalized Itô's formula}\label{generalizedIto}
For each $\lambda \in (0,1)$, consider the function $F_{\lambda}$ as defined in \eqref{approxPotential}. Then $F_{\lambda} \in C^2(\mathbb{R})$ is such that $F'_{\lambda}$ is Lipschitz-continuous and $F''_{\lambda}$ is continuous and bounded. Consider the functional $\Phi_{\lambda}$ defined as follows
\begin{equation}\label{PhiLambda}
\Phi_{\lambda}: V \to \mathbb{R}, \quad \Phi_{\lambda}(u) := \int_D F_{\lambda}(u), \quad u \in V
\end{equation}
We compute the first and second order Gâteaux derivative of $\Phi_{\lambda}$. At this purpose fix $u$, $h$, $k \in V$, then we have 
\begin{equation}\label{firstDer}
D\Phi_{\lambda}(u; h) = \lim_{t \to 0^{+}} \frac{\Phi_{\lambda}(u+t h)-\Phi_{\lambda}(u)}{t}  = \lim_{t \to 0^{+}} \frac{1}{t} \int_D F_{\lambda}(u+th)-F_{\lambda}(u) = \int_D F'_{\lambda}(u)h
\end{equation}
Indeed, exploiting the Lipschitz-continuity of $F'_{\lambda}$ and using Hölder's inequality, by the dominated convergence theorem, we obtain
\begin{align*}
& \left|\frac{1}{t}\int_D F_{\lambda}(u+th)-F_{\lambda}(u) - \int_D F'_{\lambda}(u)h \right|^2  = \left|\int_0^1\int_D (F'_{\lambda}(u+s t h) - F'_{\lambda}(u)) h ds\right|^2 \\
& \leq \|h\|_H^2 \int_0^1 \|F'_{\lambda}(u+s t h)-F'_{\lambda}(u)\|_H^2 ds \to 0 \quad \text{ as } t \to 0^{+}
\end{align*}
Hence the functional $D \Phi_{\lambda}(u)[h]: = D \Phi_{\lambda}(u;h)$ defined in \eqref{firstDer} is the Gâteaux derivative of $\phi_{\lambda}$ at $u \in V$. Since $D \Phi_{\lambda}$ is continuous as a function of $u$, thanks to \cite[Theorem 1.9]{Ambrosetti}, $D \Phi_{\lambda}(u)$ is also the Fréchet derivative of $\Phi_{\lambda}$ at $u \in V$. In particular we have that $\Phi_{\lambda} \in C^1(V)$. Proceeding in a similar way, we have 
\begin{align}\label{secondDer}
D^2 \Phi_{\lambda}(u; h, k) & = \lim_{t \to 0^{+}} \frac{D \Phi_{\lambda}(u+tk; h)- D \Phi_{\lambda}(u;h)}{t} = \lim_{t \to 0^{+}}\frac{1}{t} \int_D (F'_{\lambda}(u+tk)-F'_{\lambda}(u)) h = \int_D F''_{\lambda}(u) h k 
\end{align} 
Indeed, thanks to continuity and boundedness of $F''_{\lambda}$ and to the fact that $L^4(D) \hookrightarrow V$, using Hölder's inequality, by the dominated convergence theorem, we obtain
\begin{align*}
& \left|\frac{1}{t} \int_D (F'_{\lambda}(u+tk)-F'_{\lambda}(u)) h - \int_D F''_{\lambda}(u) h k \right|^2 = \left|\int_0^1 \int_D (F''_{\lambda}(u+stk)-F''_{\lambda}(u))hk ds\right|^2\\
& \leq \|h\|_{L^4(D)} \|k\|_{L^4(D)} \int_0^1 \|F''_{\lambda}(u+stk)-F''_{\lambda}(u)\|_H^2ds \to 0 \quad \text{ as } t \to 0^{+}
\end{align*} 
Hence the bilinear form $D^2 \Phi_{\lambda}(u)[h, k] := D^2 \Phi_{\lambda}(u;h, k) $ defined in \eqref{secondDer} is the Gâteaux derivative of $\Phi_{\lambda}$ at $u$. Since $D^2 \Phi_{\lambda}$ is continuous as a function of $u$, $D^2 \Phi_{\lambda}(u)$ is also the second order Fréchet derivative of $\Phi_{\lambda}$ at $u$ and moreover $\Phi_{\lambda} \in C^2(V)$. Finally, let us also observe that $\Phi_{\lambda}$, $D \Phi_{\lambda}$, and $D^2 \Phi_{\lambda}$ are bounded on bounded subsets of $V$ and that $D \Phi_{\lambda}$ has linear growth. Therefore Itô's formula in \cite[Theorem 4.2]{Pardoux} can be applied to function $\Phi_{\lambda}$ and gives rise to \eqref{FLambda}.

\nocite{*}
\bibliographystyle{plain}
\bibliography{ref.bib}
\end{document}